\documentclass{article}
\usepackage[english]{babel}
\usepackage{graphicx}
\usepackage{url}
\usepackage{geometry}
\usepackage{graphicx}	
\usepackage[cp1251]{inputenc}
\usepackage{amsfonts,amssymb,mathrsfs,amscd,amsmath,amsthm}
\usepackage{verbatim}

\DeclareGraphicsExtensions{.pdf,.png,.jpg}

\def\thtext#1{
  \catcode`@=11
  \gdef\@thmcountersep{. #1}
  \catcode`@=12
}

\def\threst{
  \catcode`@=11
  \gdef\@thmcountersep{.}
  \catcode`@=12
}

\theoremstyle{plain}
\newtheorem{thm}{Theorem}
\newtheorem{prop}{Proposition}[section]
\newtheorem{cor}[prop]{Corollary}
\newtheorem{ass}[prop]{Assertion}

\theoremstyle{definition}

\newtheorem{dfn}[prop]{Definition}

\newtheorem{des}[prop]{Designation}

 \catcode`@=11
 \def\.{.\spacefactor\@m}
 \catcode`@=12

\newcommand{\D}{\Delta}
\newcommand{\e}{\varepsilon}
\newcommand{\g}{\gamma}

\newcommand{\cM}{\mathcal{M}}

\renewcommand{\r}{\rho}
\newcommand{\s}{\sigma}

\newcommand{\R}{\mathbb{R}}

\newcommand{\rom}[1]{{\em #1}}
\renewcommand{\(}{\rom(}
\renewcommand{\)}{\rom)}

\renewcommand{\:}{\colon}
\newcommand{\0}{\emptyset}

\renewcommand{\c}{\circ}

\newcommand{\oPi}{\stackrel{\raise-2pt\hbox{$\c$}}\Pi}
\newcommand{\oW}{\stackrel{\raise-2pt\hbox{$\c$}}W}

\newcommand{\x}{\times}

\newcommand{\diam}{{\operatorname{diam}}}
\renewcommand{\min}{{\operatorname{min}}}

\def\dis{\operatorname{dis}}
\def\diam{\operatorname{diam}}
\def\:{\colon}
\def\opt{{\operatorname{opt}}}

\title{Convexity of Balls in Gromov--Hausdorff Space}
\date{}							
\author{Daria P.~Klibus}
\begin{document}
\maketitle
\begin{abstract}
In this paper we study the space $\mathcal{M}$ of all nonempty compact metric spaces considered up to isometry, equipped with the Gromov--Hausdorff distance. We show that each ball in $\mathcal{M}$ with center at the one-point space is convex in the weak sense, i.e., every two points of such a ball can be joined by a shortest curve that belongs to this ball; however, such a ball is not convex in the strong sense: it is not true that every shortest curve joining the points of the ball belongs to this ball. We also show that a ball of sufficiently small radius with center at a space of general position is convex in the weak sense.
\end{abstract}

\section*{Introduction}
\markright{\thesection.~Introduction}

The Gromov--Hausdorff distance was defined in 1975 by D. Edwards in article "The Structure of Superspace" \cite{Edw}, then in 1981 it was rediscovered by M.~ Gromov \cite{Gro}.

We will investigate the geometry of the space $\mathcal{M} $ of all nonempty compact metric spaces (considered up to isometry) with the Gromov--Hausdorff distance. It is well-known that the Gromov--Hausdorff distance is a metric on $\mathcal{M}$ \cite{BurBurIva}. The Gromov--Hausdorff space is Polish (complete separable) and  path-connected. Also A.O. Ivanov, N.K. Nikolaeva and A.A. Tuzhilin showed that the Gromov--Hausdorff metric is strictly intrinsic \cite{IvaNikTuz}.

The present paper is devoted to the following question: are balls in the Gromov--Hausdorff space convex. There are two concepts: convexity in the weak sense (every two points of such a ball can be joined by a shortest curve that belongs to this ball) and strong sense (each shortest curve connecting any pair of points of the set, belongs to this set). We show that a ball of nonzero radius with center at the one-point space is convex in the weak sense, but not convex in the strong one. We also show that a ball of sufficiently small radius with center at a generic position space is convex in the weak sense.

I am grateful to my scientific adviser professor Alexey A.~Tuzhilin and to professor Alexander O.~Ivanov for stating the problem and regular attention to my work.

The work was supported by the Russian Foundation for Basic Research (grant No. 16-01-00378-a) and the program ``Leading Scienti c Schools'' (grant no. NSh-6399.2018.1).

\section{Preliminaries}
\markright{\thesection.~Preliminaries}

Let $X$ be an arbitrary metric space. By $|xy|$ we denote the distance between it is two points $x$ and $y$. For any point $x \in X$ and a real number $r > 0$ we denote by $U_\e(x)= \{y \in X: |xy| < \e\}$ the \emph{open ball of radius $\varepsilon$ centered at $x$}; for each nonempty $A \subset X$ and a real number $r > 0$ we put $U_\e(A) = \cup_{a\in A} U_\e(a)$ an call it \emph{open $r$-neighborhood of the set $A$}. The \emph{closed ball of radius $\e$ centered at $x$}  is $B_\e(x) = \{y \in X: |xy| \le \e\}$. For $x \in X$ and nonempty $A \subset X$ we put $|xA| = \inf\{|xa| : a \in A\}$. For nonempty $A \subset X$ and non-negative $r$ (possibly equal $\infty$), the \emph{closed $r$-neighborhood of the set $A$} is $B_r(A) = \{x \in X : |xA|\le r\}$.

\begin{dfn}
Let $X$ and $Y$ be two nonempty subsets of a metric space. The \emph{Hausdorff distance between $X$ and $Y$} is $\text{$d$}_{H}(X, Y)  = \inf\left\{\text{$\e>0 \mid ($U$_\e(X) \supset Y) \  \& \  ($U$_\e(Y) \supset X)$}\right\}.$
\end{dfn}

By $\mathcal{H}(X)$ we denote the family of all nonempty closed bounded subsets of a metric space $X$.

\begin{prop}[\cite{BurBurIva}]
The function $d_{H}$ is a metric on $\mathcal{H}(X)$.
\end{prop}

\begin{prop}[\cite{IvaTu}]
 For any metric space $X$, any $A \in \mathcal{H}(X)$, and any nonnegative $r$ we have $B_r(A) \in \mathcal{H}(X)$.
\end{prop}

\begin{dfn}[\cite{IvaTu}]
Let $W$ be any metric space, $a,b \in W$, $|ab| = r$, $s \in [0,r]$. A point $c \in W$ is \emph{in $s$-position between $a$ and $b$}, if $|ac| = s$ and $|cb| = r - s $.
\end{dfn}

\begin{prop}[\cite{IvaTu}]\label{2}
Let $X$ be any metric space and $A,B \in \mathcal{H}(X)$, $r = d_H(A,B), s \in [0,r]$. If a set $C \in \mathcal{H}(X)$ is in $s$-position between $A$ and $B$, then $C \subset B_s(A) \cap B_{r-s}(B)$.
\end{prop}

\begin{des}[\cite{IvaTu}]\label{4}
In what follows, we denote the set $B_s(A)\cap B_{r-s}(B)$  by $C_s(A,B)$.
\end{des}

\begin{dfn}
Let $X$ and $Y$ be metric spaces. A triple $(X',Y',Z)$ that consists of a metric space $Z$ and it is subsets $X'$ and $Y'$ isometric to $X$ and $Y$, respectively, is called a \emph{realization of the pair $(X,Y)$}. \emph{The Gromov--Hausdorff distance $d_{GH}(X,Y)$ between $X$ and $Y$} is the infimum of real numbers $\r$ such that there exists a realization $(X',Y',Z)$ of the pair $(X,Y)$ with $d_{H}$$(X',Y') \le \r$.
\end{dfn}

By $\mathcal{M}$ we denote the set of all compact metric spaces, considered up to an isometry, with the Gromov--Hausdorff distance. The restriction of $d_{GH}(X,Y)$ onto $\mathcal{M}$ is a metric \cite{BurBurIva}.

A metric on a set $X$ is \emph{stricly intrinsic}, if any two points $x$, $y \in X$ are joined by a curve whose length is equal to the distance between $x$ and $y$ (this curve is called \emph{shortest}).

\begin{thm}[\cite{IvaTu}]
Let $X$ be a complete locally compact space with intrinsic metric. Then for any $A, B \in \mathcal{H}(X)$, $r = d_H(A,B), s \in [0,r]$, the set $C_s(A,B)$ belongs $\mathcal{H}(X)$ and is in $s$-position between $A$ and $B$.
\end{thm}

\begin{cor}[\cite{BurBurIva}]
Let $X$ be a complete locally compact space with intrinsic metric \($X$ is boundedly compact with stricly intrinsic metric $[1]$\/\). Then $\mathcal{H}(X)$ is boundedly compact, and Hausdorff metric is stricly intrinsic.
\end{cor}

\begin{cor}[\cite{IvaTu}]\label{7}
Let $X$ be a complete locally compact space with intrinsic metric, $A,B \in \mathcal{H}(X)$, and $r = d_H(A,B)$. Then $\gamma(s) = C_s(A,B)$, $s \in [a,b]$ is a shortest curve connecting $A$ and $B$, where the length of curve $\gamma$ is equal to $d_H(A,B)$, and the parameter $s$ is natural.
\end{cor}

\begin{dfn}
A nonempty subset $M$ of a metric space $X$ with stricly intrinsic metric is \emph{convex in the weak sense}, if for any two points from $M$, some shortest curve connecting them belongs to $M$.
\end{dfn}

\begin{dfn}
A nonempty subset $M$ of a metric space $X$ with stricly intrinsic metric is \emph{convex in the strong sense}, if for any two points from $M$, any shortest curve connecting them belongs to $M$.
\end{dfn}

Let $X$ and $Y$ be arbitrary nonempty sets. Recall that a \emph{relation} between the sets $X$ and $Y$ is a subset of the Cartesian product $X \times Y$. By $\mathcal{P}(X, Y)$ we denote the set of all nonempty relations between $X$ and $Y$. Let $\pi_X\: (X, Y)\to X$ and $\pi_Y\: (X, Y)\to Y$ be the canonical projections, i.e., $\pi_X\ (x, y) = x$ and $\pi_Y\ (x, y) = y$. In the same way we denote the restrictions of the canonical projections to each relation $\sigma \in \mathcal{P}(X, Y)$.

Let us consider each relation $\sigma \in \mathcal{P}(X, Y)$  as a multivalued mapping whose domain may be less than $X$. Then, similarly with the case of mappings, for any $x\in X$ and any $A\subset X$  their images $\sigma(x)$ and $\sigma(A)$ are defined, and for any $y \in Y$ and any $B \subset Y$ their preimages $\sigma ^{-1}(y)$ and $\sigma ^{-1}(B)$ are also defined.

\begin{dfn}
A relation $R \subset X \times Y$ between $X$ and $Y$ is called a \emph{correspondence}, if the restrictions of the canonical projections $\pi_X$ and $\pi_Y$ onto $R$ are surjective. By $\mathcal{R}(X, Y)$ we denote the set of all correspondences between $X$ and $Y$.
\end{dfn}

\begin{dfn}
Let $X$ and $Y$ be arbitrary metric spaces. The \emph{distortion} $\dis\sigma$ of a \emph{relation} $\sigma\in \mathcal{P}(X, Y)$ is the value
$$
 \dis\s=\sup\Bigl\{\bigl| |xx'| - |yy'| \bigr|:(x,y),\,(x',y')\in\s\Bigr\}.
$$
\end{dfn}

\begin{prop}[\cite{BurBurIva}]\label{1,7} For any metric spaces $X$ and $Y$ we have
$$
d_{GH}(X,Y) = \frac{1}{2}\inf\left\{\text{$\dis R: R\in \mathcal{R}(X, Y)$}\right\}.
$$
\end{prop}

\begin{dfn}
A relation $R \in \mathcal{R}(X, Y)$ is called \emph{optimal}, if $d_{GH}(X,Y)$ = $\frac{1}{2}\dis\ R$. The set of all optimal correspondences between $X$ and $Y$ is denoted by $\mathcal{R}_{\opt}(X, Y)$.
\end{dfn}

\begin{prop}[\cite{IvaIliTuz}]\label{3}
For any $X, Y \in \mathcal{M}$ we have $\mathcal{R}_{\opt}(X, Y) \ne \0$.
\end{prop}

\begin{prop}[\cite{IvaIliTuz}]\label{1,9}
For any $X$, $Y\in \mathcal{M}$ and each $R \in \mathcal{R}_{\opt}(X, Y)$ the family $R_t$, $t\in [0, 1]$, of compact metric spaces such that $R_0=X$, $R_1=Y$, and for $t\in(0, 1)$ the space $R_t$ is equal to $(R, \rho_t)$, where $\rho_t\bigl((x, y),(x', y')\bigr) = (1-t)\bigl| xx'\bigr|+t\bigl|yy'\bigr|$, is a shortest curve in $\mathcal{M}$ connecting $X$ and $Y$.
\end{prop}
For a metric space $X$ by $\diam(X)$ we denote its \emph{diameter}:
$$
\diam(X) = \sup_{x,x' \in X}|xx'|.
$$
Let $\D_1$ be a single-point space.
\begin{ass}[\cite{BurBurIva}]\label{8}
For any metric space $X$ we have $d_{GH}(X,\D_1)$ = $\diam(X)/2.$
\end{ass}
\begin{dfn}

We say that a finite metric space $M$ is \emph{in general position}, or is a \emph{space of general position}, if all its nonzero distances are distinct, and all triangle inequalities are strict.

\end{dfn}
For a metric space $X$ we define the following values:
$$
s(X) = \inf\{|xy| : x \ne y\}, \ e(X) = \inf\Bigl\{\Bigl||xy| - |zw|\Bigr|: x\ne y, z \ne w, \{x, y\} \ne \{z, w\} \Bigr\}.
$$

\begin{prop}[\cite{IvaTuz}]\label{5}
Let $M = \{1,\ldots,n\}$ be a metric space. Then for any $0 < \e \le s(M)/2$ and each $X \in\mathcal{M}$ such that $2d_{GH}(M,X) < \e$, there exists a partition $X = \sqcup_{i=1}^n X_i$ unique up to numeration by points of $M$, possessing the following properties\/\rom:
\begin{enumerate}
\item $\diam X_i < \e$\rom;
\item for any $i, j \in M$ and any $x \in X_i$ and $x' \in X_j$ \rom(here the indices $i$ and $j$ may be equal to each other\/\rom) it holds $\Bigl||xx'| - |ij|\Bigr| < \e$.
\end{enumerate}
\end{prop}

\begin{prop}[\cite{IvaTuz}]
Let $M = \{1,\ldots,n\}$ be a metric space. Then for any $0 < \e \le s(M)/2$, any $X \in \mathcal{M}$, $2d_{GH}(M,X) < \e$, and each $R \in \mathcal{R}_{\opt}(M,X)$ the family $\{R(i)\}_{i=1}^n$ is a partition of the set $X$, satisfying the following properties\/\rom:
\begin{enumerate}
\item $\diam X_i < \e$\rom;
\item  for any $i, j \in M$, $x \in R(i), x'\in R(j)$ it holds $\Bigl||xx'| - |ij|\Bigr| < \e$.
\end{enumerate}
Moreover, if $R'$is another optimal correspondence between $M$ and $X$, then the partitions $\{R(i)\}_{i=1}^n$ and $\{R'(i)\}_{i=1}^n$ may differ from each other only by numerations generated by the correspondences $i \mapsto R(i)$ and $i \mapsto R'(i)$.
\end{prop}

\begin{dfn}[\cite{IvaTuz}]
The family $\{X_i\}$ from Proposition \ref{5} we call the \emph{canonical partition} of the space $X$ with respect to $M$.
\end{dfn}

\begin{prop}[\cite{IvaTuz}]\label{6}
Let $M = \{1,\ldots,n\}$ be a metric space, $n \ge 3$, $e(M) > 0$. Choose an arbitrary $0 < \e \le \frac{1}{4}\min\{s(M), e(M)\}$, any $X, Y \in \mathcal{M}$, $2d_{GH}(M,X) < \e$, $2d_{GH}(M, Y ) < \e$, and let $\{X_i\}$ and $\{Y_i\}$ denote the canonical partitions of $X$ and $Y$, respectively, w.r.t. $M$. Then for each $R \in \mathcal{R}_{\opt}(X, Y )$ there exist $R_i \in R(X_i, Y_i)$ such that $R = \sqcup_{i=1}^n R_i$.

\end{prop}

\section{The main results}
\markright{\thesection.~The main results}

\begin{thm}
A ball with center at the one-point metric space is convex in the weak sense.
\end{thm}
\begin{proof}
Let $B=B_r(\D_1)$ be a closed ball with center at $\D_1$, where $r>0$, and $X,\;Y \in B$. By Proposition \ref{8},
$$
d_{GH}(\D_1,X) = \diam(X)/2\le r, \ d_{GH}(\D_1,Y) = \diam(Y)/2\le r.
$$
We choose some correspondence $R\in\mathcal{R}_{\opt}(X, Y)$ which exists by Proposition \ref{3}. We construct the space $R_t =(R, \rho_t)$ with metric $\rho_t\bigl((x, y),(x', y')\bigr) = (1-t)\bigl| xx'\bigr|+t\bigl|yy'\bigr|$, where $t\in(0, 1)$, and put $R_0=X$, $R_1=Y$. Then, by Proposition~\ref{1,9}, the curve $R_t$, $t\in[0,1]$, connecting $X$ and $Y$ is shortest.
\\We show that the curve $R_t$ lies in the ball $B$. To do that, we estimate the Gromov--Hausdorff distance between the center $\D_1$ and the space $R_t$. We have
$$
d_{GH}(\D_1, R_t) = \diam(R_t)/2. \;
$$
For any $x,x'\in X$ and $y,y'\in Y$ it holds
$$
|xx'| \le \diam X \le \max(\diam X, \diam Y); \; |yy'| \le \diam Y\le \max(\diam X, \diam Y).
$$
Therefore,
$$
\diam(R_t) = \max|(x, y) (x', y')|_{\rho_t} = \max \bigl((1-t)|xx'|+t|yy'|\bigr) \le
$$
$$
\le (1-t)\max(\diam X, \diam Y)+t\max(\diam X, \diam Y) = \max(\diam X, \diam Y),
$$
\\
thus, $d_{GH}(\D_1,R_t)\le \frac1 2 \max(\diam X, \diam Y)=\max\bigl(d_{GH}(\D_1,X),d_{GH}(\D_1,Y)\bigr)\le r$.
\end{proof}
\
\begin{thm}
A ball with center at the one-point metric space is not convex in the strong sense.
\end{thm}
\begin{proof}
To prove that, we construct a shortest curve connecting some spaces $A, B \in B_r(\Delta_1)\subset\cM$, but not containig in $B_r(\D_1)$. Let $A = [0, 2r]\subset\R$ and $B = \{0, 2r\}\subset\R$. We choose some correspondence $R\in\mathcal{R}_{\opt}(B, A)$ (it exists by Proposition \ref{3}) and estimate it:
\begin{multline*}
\dis R = \sup \Bigr\{|aa'|: a, a'\in R(0); \ |aa'|: a, a'\in R(2r); \ \bigl|2r - |aa'|\bigr|: a\in R(0), a'\in R(2r)\Bigr\}=
\end{multline*}
$$
 =\sup\Bigr\{\diam R(0),\diam R(2r),\bigl|2r - |aa'|\bigr|: a\in R(0), a'\in R(2r) \Bigr\}\le 2r.
$$
1) If $R(0)\cap R(2r) \ne \emptyset $, then choosing $a=a'\in R(0)\cap R(2r)$, we have $\dis R = 2r$.
\\ 2) If $R(0)\cap R(2r) = \emptyset $, then for any $\e>0$ there exist $a\in R(0),\; a'\in R(2r)$, such that $|aa'|<\e$, thus $\dis R = 2r$.
\\Then, by Proposition \ref{1,7}, we have $d_{GH}(A, B) = r$.
Since $d_H(A,B)=r$, it holds $d_{GH}(A,B)=d_H(A,B)$. For $t\in[0,r]$ we put $\g(t)=C_t(A,B)=B_t(A)\cap B_{r-t}(B)$. Applying the corollary \ref{7}, we see, that $\g(t)$ is a shortest curve in $\mathcal{H}(\R)$.
\\For any partition $t_0=0<t_1<\cdots<t_n=r$ of the segment $[0,r]$ we have
$$d_{GH}(A,B)\le\sum_{i=1}^nd_{GH}\bigr(\g(t_{i-1}),\g(t_i)\bigr)\le\sum_{i=1}^nd_H \bigr(\g(t_{i-1}),\g(t_i)\bigr)=d_H(A,B)=d_{GH}(A,B),$$
so
$$d_{GH}(A,B)=\sum_{i=1}^nd_{GH} \bigr(\g(t_{i-1}),\g(t_i)\bigr).$$
Since the length of the curve $\g$ is equal to the supremum of the sums $\sum_{i=1}^nd_{GH}\bigr(\g(t_{i-1}),\g(t_i)\bigr)$ over all possible partitions of the segment $[0,r]$, and all these sums are the same and equal $d_{GH}(A,B)$, then the length of the curve $\g$ is equal to $d_{GH}(A,B)$, therefore $\g$ is a shortest curve.
\\We show that this curve does not lie entirely in the ball $B_r(\Delta_1)$. To do that we calculate $d_{GH}\bigr(C_t(A,B),\D_1\bigr)$. By Assertion \ref{8}, $d_{GH}\bigr(C_t(A,B),\D_1\bigr) = \frac{\diam\bigr(C_t(A,B)\bigr)}{2}$. Notice that for $t=\frac{r}{2}$ we have $\diam\bigr(C_{\frac{r}{2}}(A,B)\bigr)= 3r$, so $d_{GH}\bigr(C_{r/2}(A,B),\Delta_1\bigr)=3r/2>r$, thus $\g(r/2)\not\in B_r(\Delta_1)$.
\end{proof}

\begin{thm}
For any space $M \in \mathcal{M}$ in general position and any $0 < r \le \frac{1}{4}\min\{s(M), e(M)\}$ the ball with center at $M$ and radius $r$ is convex in the weak sense.
\end{thm}
\begin{proof}
Let $\varepsilon=2r$, $M = \{1,\ldots,n\}$, and $X, \ Y \in B_{\e/2}(M)$. By Proposition \ref{5}, there exist unique (up to a numeration of points of $M$) partitions $X = \sqcup_{i=1}^n X_i$ and $Y = \sqcup_{i=1}^n Y_i$, possessing the following properties: for any $x_i\in X_i, x_j\in X_j, y_i\in Y_i, y_j\in Y_j $ it holds $\Bigl||x_i x_j| - |ij|\Bigr| < \e$ and $\Bigl||y_i y_j| - |ij|\Bigr| < \e$. By Proposition \ref{6}, for any $R \in \mathcal{R}_{\opt}(X, Y)$ there are $R_i \in R(X_i, Y_i)$, where $R = \sqcup_{i=1}^n R_i$. We choose some correspondence $R\in\mathcal{R}_{\opt}(X, Y)$ (it exists by Proposition \ref{3}, it exists). \\We construct a shortest curve $R_t$, as in Proposition \ref{1,9}. To prove convexity in the weak sense, we show that $d_{GH}(M,R_t)\le {\e/2}$. Let us introduce a the correspondence $R'\in \mathcal{R}(M,R_t)$ us $R'=\sqcup_{i=1}^n\{i\}\x R_i$. We have
\begin{multline*}d_{GH}(M,R_t)\le\frac{1}{2}\dis R' = \frac{1}{2}\sup\Bigl\{\bigl| |ij| - |p_i p_j|_t \bigr|: i, j \in M, \ (i,p_i),\,(j,p_j)\in R'\Bigr\}=
\end{multline*}
$$ =\frac{1}{2}\sup\Bigl\{\bigl| |ij| - (1-t)|x_i x_j|-t|y_i y_j| \bigr|:\ i,j\in M,\ (x_i,y_i)=p_i\in R_i,
(x_j,y_j)=p_j\in R_j\Bigr\}=$$
$$ =\frac{1}{2}\sup\Bigl\{\bigl| (1 - t) |ij|+t|ij| - (1-t)|x_i x_j| - t|y_i y_j| \bigr|\Bigr\}= $$
$$=\frac{1}{2}\sup\Bigl\{\bigl| (1 - t)(|ij| - |x_i x_j|) +  t(|ij| - |y_i y_j|) \bigr|\Bigr\}\le$$
$$ \le\frac{1}{2}(1-t)\sup\Bigl\{\bigl||ij| - |x_i x_j|\bigr|\Bigr\} + \frac{1}{2}t\sup\Bigl\{\bigl||ij| - |y_i y_j|\bigr|\Bigr\}\le \frac{1}{2}(1-t)\e + \frac{1}{2}t\e = \frac{\e}{2}.$$
\end{proof}

\end{document}